\documentclass[letterpaper,10pt]{amsart}

\usepackage[all]{xy}                        %

\CompileMatrices                            % Faster

\UseTips                                    % Use

\input xypic
\usepackage[bookmarks=true]{hyperref}       % Hyperref
\usepackage{amssymb,latexsym,amsmath,amscd}
\usepackage{xspace}
\usepackage{color}
\usepackage{kpfonts}
\usepackage{graphicx}
\usepackage{multicol}
\reversemarginpar

\theoremstyle{plain}
\newtheorem{theorem}{Theorem}[section]
\newtheorem*{theorem*}{Theorem}
\newtheorem{proposition}[theorem]{Proposition}
\newtheorem{corollary}[theorem]{Corollary}
\newtheorem{lemma}[theorem]{Lemma}

\theoremstyle{definition}
\newtheorem{definition}[theorem]{Definition}

\newtheorem{remark}[theorem]{Remark}
\newtheorem{example}[theorem]{Example}

\newcommand{\enm}[1]{\ensuremath{#1}}          %
\newcommand{\op}[1]{\operatorname{#1}}
\newcommand{\cal}[1]{\mathcal{#1}}

\newcommand{\sF}{\enm{\mathcal{F}}}

\newcommand{\CC}{\enm{\mathbb{C}}}

\newcommand{\ZZ}{\enm{\mathbb{Z}}}

\newcommand{\PP}{\enm{\mathbb{P}}}

\newcommand{\Aa}{\enm{\cal{A}}}
\newcommand{\Bb}{\enm{\cal{B}}}
\newcommand{\Cc}{\enm{\cal{C}}}

\newcommand{\Ee}{\enm{\cal{E}}}
\newcommand{\Ff}{\enm{\cal{F}}}
\newcommand{\Gg}{\enm{\cal{G}}}

\newcommand{\Mm}{\enm{\cal{M}}}

\newcommand{\Oo}{\enm{\cal{O}}}

\newcommand{\Vv}{\enm{\cal{V}}}

\renewcommand{\phi}{\varphi}
\renewcommand{\theta}{\vartheta}
\renewcommand{\epsilon}{\varepsilon}

\newcommand{\Ext}{\op{Ext}}

\renewcommand{\to}[1][]{\xrightarrow{\ #1\ }}

\newcommand{\old}[1]{}

\begin{document}

\title[Ulrich bundles on the degree six Segre fourfold ]{Ulrich bundles on the degree six Segre fourfold}

\author{F. Malaspina }

\address{Politecnico di Torino, Corso Duca degli Abruzzi 24, 10129 Torino, Italy}
\email{francesco.malaspina@polito.it}



\begin{abstract}

We completely characterize the bigraded resolutions
of Ulrich bundles of arbitrary rank on the Segre fourfold $\PP^2\times\PP^2$.
We characterize the Ulrich  bundles $\Vv$ of arbitrary rank on $\PP^2\times\PP^2$ with $h^1(\Vv\otimes\Omega\boxtimes\Omega)=0$ or with $h^2(\Vv\otimes\Omega(-1)\boxtimes\Omega(-1))=0$ or obtained as pullback from $\PP^2$ and we construct more complicated examples.
\end{abstract}

\maketitle

%\tableofcontents

\section{Introduction}
A locally free sheaf (or ``bundle'') $\Ee$ on a projective varity $X$ is aCM if it has not intermediate cohomology or if the module
$E$ of global sections of
$\Ee$ is a maximal Cohen-Macaulay module.
There has been increasing interest on the classification of  aCM bundles on various projective varieties, which is important in a sense that the aCM bundles are considered to give a measurement of complexity of the underlying space. Moreover the understanding of aCM bundles is crucial for the study of any bundle on $X$ as it is showed on \cite{MRa}.  A special type of aCM sheaves, called the Ulrich sheaves, are the ones achieving the maximum possible minimal number of generators. These bundles
are characterized by the linearity of the minimal graded free resolution
over the polynomial ring of their module of global section.
Ulrich bundles, originally studied for computing Chow forms,
conjecturally exist over any variety (see \cite{ES}). But the conjecture has been checked only for a few varieties, e.g. in case of surfaces, del Pezzo surfaces, rational normal scrolls, rational aCM surfaces in $\PP^4$, ruled surfaces  and so on; see  \cite{CKM, MR, MP, ACM}. The case of Veronese varieties has been studied in \cite{ES}, \cite{CG} and \cite{L} and the case of Hirzebruch surfaces in \cite{A}. Although there are some occasions where the classification problem of Ulrich bundles of special type is done as in \cite
{AHMP, CCHMW, CM1, CM2}, the completion of classification problem is difficult in usual.

In \cite{CMP} and \cite{MR} it shown that Segre varieties and rational normal scroll (except $\PP^1\times\PP^1$ and cubic and quartic scroll, see \cite{FM} and \cite{FP}) are of Ulrich wild representation type namely they support families of Ulrich bundles of arbitrary large. The representation type  is determined by considering a certain type of family of Ulrich bundles. In \cite{AHMP} in given a classification of Ulrich vector bundles of arbitrary rank on rational normal scroll. A consequence of the main result is that the moduli spaces of Ulrich bundles are zero-dimensional. The case $\PP^2\times\PP^1$ has been studied in more details in \cite{CFMS}, where all the aCM bundles are classified. The first biprojective space which is not a rational scroll is $\PP^2\times\PP^2$. Here the families of Ulrich bundles are much more complicated as can be deduced from the study of the rank two case achieved in \cite{CFM}.

 In this article we pay our attention to the case of arbitrary rank. In order to do that we show that every Ulrich bundle $\Vv$ on $\PP^2\times\PP^2$ is regular according to both the two different notions of Castelnuovo-Mumford regularity given in \cite{bm2} and \cite{hw}. Then we compute the cohomology of $\Vv$ tensored with $\Oo_{\PP^2}\boxtimes\Omega$, $\Omega\boxtimes\Oo_{\PP^2}$ and $\Omega\boxtimes\Omega$  with suitable twists. In particular we obtain that $\Vv$ is of natural cohomology (as in \cite{ES} on Veronese varities) in a suitable range. More precisely we get that the only nonzero cohomology for $\Vv(m,n)$ with $|m-n|\leq 1$ may be given by $$\textrm{$a_i=h^i(\Vv(-i-1,-i))$ and $b_i=h^i(\Vv(-i,-i-1)).$}$$   At this point we choose suitable full exceptional collections in order to apply a Beilinson type spectral sequence and to obtain the following resolution
   $$0\to\Oo_X(-1,0)^{\oplus a_2}\oplus\Oo_X(0,-1)^{\oplus b_2}\to\Oo_X(-1,1)^{\oplus a_1}\oplus\Oo_X^{\oplus 3a_2+3b_2}\oplus\Oo_X(1,-1)^{\oplus b_1}\to$$
$$\to\Oo_X(1,0)^{\oplus a_0}\oplus\Oo_X(0,1)^{\oplus b_0}\to \Vv\to 0.$$
  Moreover we  characterize the Ulrich  bundles $\Vv$ of arbitrary rank on $\PP^2\times\PP^2$ with $h^1(\Vv\otimes\Omega\boxtimes\Omega)=0$ or with $h^1(\Vv\otimes\Omega(-1)\boxtimes\Omega(-1))=0$ or obtained as pullback from $\PP^2$ and we construct more complicated examples.  We believe that the study of this variety will be the key step for the understanding of the families of Ulrich bundles over all the biprojective spaces.

Here we summarize the structure of this article. In section \ref{sec2} we introduce the definition of Ulrich bundles and several notions in derived category of coherent sheaves to understand the Beilinson spectral sequence. In section \ref{sec3} we recall the definition of Castelnuovo-Mumford given in \cite{bm2} and \cite{hw} and we made the cohomological computations. In section \ref{sec4} we study the examples of families of Ulrich bundles and we prove the main results.
In section \ref{sec5} we discuss the case of the hyperplane section of $\PP^2\times\PP^2$: the flag variety $F(0,1,2)$.

The author wants to thank M. Aprodu and P. Rao for helpful discussions on the subject.

%%%%%%%%%%%%%%%%%%%%%%%

\section{Preliminaries}\label{sec2}
Throughout the article our base field is the field of complex numbers $\CC$.

\begin{definition}
A coherent sheaf $\Ee$ on a projective variety $X$ with a fixed ample line bundle $\Oo_X(1)$ is called {\it arithmetically Cohen-Macaulay} (for short, aCM) if it is locally Cohen-Macaulay and $H^i(\Ee(t))=0$ for all $t\in \ZZ$ and $i=1, \ldots, \dim (X)-1$.
\end{definition}

%If $\Ee$ is a coherent sheaf on a projective scheme $X$, then we may consider its Hilbert polynomial $\mathrm{P}_{\Ee}(m)\in \QQ [m]$ with the leading coefficient $\mu(\Ee)/d!$, where $d$ is the dimension of $\mathrm{Supp}(\Ee)$ and $\mu=\mu(\Ee)$ is called the {\it multiplicity} of $\Ee$.

%\begin{definition}\label{def}
%If $\dim \mathrm{Supp}(\Ee)=\dim (X)$, then the {\it rank} of $\Ee$ is defined to be
%$$\mathrm{rank}(\Ee)=\frac{\mu(\Ee)}{\mu(\Oo_X)}.$$
%Otherwise it is defined to be zero.
%\end{definition}
%For an integral scheme $X$, the rank of $\Ee$ is the dimension of the stalk $\Ee_x$ at the generic point $x\in X$. But in general $\mathrm{rank}(\Ee)$ need not be integral.

\begin{definition}
For an {\it initialized} coherent sheaf $\Ee$ on $X$, i.e. $h^0(\Ee(-1))=0$ but $h^0(\Ee)\ne 0$, we say that $\Ee$ is an {\it Ulrich sheaf} if it is aCM and $h^0(\Ee)=\deg (X)\mathrm{rank}(\Ee)$.
\end{definition}

Given a smooth projective variety $X$, let $D^b(X)$ be the the bounded derived category of coherent sheaves on $X$. An object $E \in D^b(X)$ is called {\it exceptional} if $\Ext^\bullet(E,E) = \mathbb C$.
A set of exceptional objects $\langle E_0, \ldots, E_n\rangle$ is called an {\it exceptional collection} if $\Ext^\bullet(E_i,E_j) = 0$ for $i > j$. An exceptional collection is said to be {\it full} when $\Ext^\bullet(E_i,A) = 0$ for all $i$ implies $A = 0$, or equivalently when $\Ext^\bullet(A, E_i) = 0$ does the same.

\begin{definition}\label{def:mutation}
Let $E$ be an exceptional object in $D^b(X)$.
Then there are functors $\mathbb L_{E}$ and $\mathbb R_{E}$ fitting in distinguished triangles
$$
\mathbb L_{E}(T) 		\to	 \Ext^\bullet(E,T) \otimes E 	\to	 T 		 \to	 \mathbb L_{E}(T)[1]
$$
$$
\mathbb R_{E}(T)[-1]	 \to 	 T 		 \to	 \Ext^\bullet(T,E)^* \otimes E	 \to	 \mathbb R_{E}(T)	
$$
The functors $\mathbb L_{E}$ and $\mathbb R_{E}$ are called respectively the \emph{left} and \emph{right mutation functor}.
\end{definition}

%The collections given by
%\begin{align*}
%E_i^{\vee} &= \mathbb L_{E_1} \mathbb L_{E_2} \dots \mathbb L_{E_{n-i}} E_{n-i+1};\\
%^\vee E_i &= \mathbb R_{E_n} \mathbb R_{E_{n-1}} \dots \mathbb R_{E_{n-i+2}} E_{n-i+1},
%\end{align*}
%are again full and exceptional and are called the \emph{left} and \emph{right dual} collections; the dual of a coherent sheaf $\Ee$ in the usual sense will be denoted by $\Ee^*$. The dual collections are characterized by the following property; see \cite[Section 2.6]{GO}.
%\begin{equation}\label{eq:dual characterization}%Just once
%\Ext^k(^\vee E_i, E_j) = \Ext^k(E_i, E_j^\vee) = \left\{
%\begin{array}{cc}
%\mathbb C & \textrm{\quad if $i+j = n+1$ and $i = k + 1$} \\
%0 & \textrm{\quad otherwise}
%\end{array}
%\right.
%\end{equation}

The collections given by
\begin{align*}
E_i^{\vee} &= \mathbb L_{E_0} \mathbb L_{E_1} \dots \mathbb L_{E_{n-i-1}} E_{n-i};\\
^\vee E_i &= \mathbb R_{E_n} \mathbb R_{E_{n-1}} \dots \mathbb R_{E_{n-i+1}} E_{n-i},
\end{align*}
are again full and exceptional and are called the \emph{right} and \emph{left dual} collections. The dual collections are characterized by the following property; see \cite[Section 2.6]{GO}.
\begin{equation}\label{eq:dual characterization}%Just once
\Ext^k(^\vee E_i, E_j) = \Ext^k(E_i, E_j^\vee) = \left\{
\begin{array}{cc}
\mathbb C & \textrm{\quad if $i+j = n$ and $i = k $} \\
0 & \textrm{\quad otherwise}
\end{array}
\right.
\end{equation}

%When a full exceptional collection is available, one has a generalized Beilinson spectral sequence.
\begin{theorem}[Beilinson spectral sequence]\label{thm:Beilinson}
Let $X$ be a smooth projective variety and with a full exceptional collection $\langle E_0, \ldots, E_n\rangle$ of objects for $D^b(X)$. Then for any $A$ in $D^b(X)$ there is a spectral sequence
%in the square $-n\leq p\leq 0, 0\leq q\leq n$
with the $E_1$-term
\[
E_1^{p,q} =\bigoplus_{r+s=q} \Ext^{n+r}(E_{n-p}, A) \otimes \mathcal H^s(E_p^\vee )
\]
which is functorial in $A$ and converges to $\mathcal H^{p+q}(A)$.
\end{theorem}

The statement and proof of Theorem \ref{thm:Beilinson} can be found both in  \cite[Corollary 3.3.2]{RU}, in \cite[Section 2.7.3]{GO} and in \cite[Theorem 2.1.14]{BO}.

%We will use the following version of the above theorem:

Let us assume next that the full exceptional collection  $\langle E_0, \ldots, E_n\rangle$ contains only pure objects of type $E_i=\mathcal E_i^*[-k_i]$ with $\mathcal E_i$ a vector bundle for each $i$, and moreover the right dual collection $\langle E_0^\vee, \ldots, E_n^\vee\rangle$ consists of coherent sheaves. Then the Beilinson spectral sequence is much simpler since
\[
E_1^{p,q}=\Ext^{n+q}(E_{n-p}, A) \otimes E_p^\vee=H^{n+q+k_{n-p}}(\mathcal E_{n-p}\otimes A)\otimes E_p^\vee.
\]

Note however that the grading in this spectral sequence applied for the projective space is slightly different from the grading of the usual Beilison spectral sequence, due to the existence of shifts by $n$ in the index $p,q$. Indeed, the $E_1$-terms of the usual spectral sequence are $H^q(A(p))\otimes \Omega^{-p}(-p)$ which are zero for positive $p$. To restore the order, one needs to change slightly the gradings of the spectral sequence from Theorem \ref{thm:Beilinson}. If we replace, in the expression
\[
E_1^{u,v} = \mathrm{Ext}^{v}(E_{-u},A) \otimes E_{n+u}^\vee=
H^{v+k_{-u}}(\mathcal E_{-u}\otimes A) \otimes \mathcal F_{-u}
\]
$u=-n+p$ and $v=n+q$ so that the fourth quadrant is mapped to the second quadrant, we obtain the following version (see \cite{AHMP}) of the Beilinson spectral sequence:

%\begin{theorem}\label{use}
%Let $X$ be a smooth projective variety with a full exceptional collection
%$\langle E_0, \ldots, E_n\rangle$
%%$\langle E_0^*[-k_0], \ldots, E_n^*[-k_n]\rangle$
%where $E_i=\mathcal E_i^*[-k_i]$ with each $\mathcal E_i$ a vector bundle and $(k_0, \ldots, k_n)\in \ZZ^{\oplus n+1}$ such that there exists a sequence $\langle F_n=\mathcal F_n, \ldots, F_0=\mathcal F_0\rangle$ of vector bundles satisfying
%\begin{equation}\label{order}%Just once
%\mathrm{Ext}^k(E_i,\mathcal F_j)=H^{k-k_i}( \mathcal E_i\otimes \mathcal F_j) =  \left\{
%\begin{array}{cc}
%\mathbb C & \textrm{\quad if $i=j=k$} \\
%0 & \textrm{\quad otherwise}
%\end{array}
%\right.
%\end{equation}
%i.e. the collection $\langle F_n, \ldots, F_0\rangle$ labelled in the reverse order is the right dual collection of $\langle E_0, \ldots, E_n\rangle$.
%Then for any coherent sheaf $A$ on $X$ there is a spectral sequence in the square $0\leq u\leq n$, $-n\leq v\leq 0$ with the $E_1$-term
%\[
%E_1^{u,v} = \mathrm{Ext}^v(E_{n-u},A) \otimes \mathcal F_{n-u}=
%H^{v-k_{n-u}}(\mathcal E_{n-u}\otimes A) \otimes \mathcal F_{n-u}
%\]
%which is functorial in $A$ and converges to \begin{equation}
%E_{\infty}^{u,v}= \left\{
%\begin{array}{cc}
%A & \textrm{\quad if $u+v=0$} \\
%0 & \textrm{\quad otherwise.}
%\end{array}
%\right.
%\end{equation}
%\end{theorem}

%For practical reasons, it is more convenient to work with a second-quadrant spectral sequence instead of a fourth-quadrant spectral sequence, and this can be realised with a shift of indices $p = u-n$ and $q = v+n$. In particular, we obtain

\begin{theorem}\label{use}
Let $X$ be a smooth projective variety with a full exceptional collection
$\langle E_0, \ldots, E_n\rangle$
%$\langle E_0^*[-k_0], \ldots, E_n^*[-k_n]\rangle$
where $E_i=\mathcal E_i^*[-k_i]$ with each $\mathcal E_i$ a vector bundle and $(k_0, \ldots, k_n)\in \ZZ^{\oplus n+1}$ such that there exists a sequence $\langle F_n=\mathcal F_n, \ldots, F_0=\mathcal F_0\rangle$ of vector bundles satisfying
\begin{equation}\label{order}%Just once
\mathrm{Ext}^k(E_i,F_j)=H^{k+k_i}( \mathcal E_i\otimes \mathcal F_j) =  \left\{
\begin{array}{cc}
\mathbb C & \textrm{\quad if $i=j=k$} \\
0 & \textrm{\quad otherwise}
\end{array}
\right.
\end{equation}
i.e. the collection $\langle F_n, \ldots, F_0\rangle$ labelled in the reverse order is the right dual collection of $\langle E_0, \ldots, E_n\rangle$.
Then for any coherent sheaf $A$ on $X$ there is a spectral sequence in the square $-n\leq p\leq 0$, $0\leq q\leq n$  with the $E_1$-term
\[
E_1^{p,q} = \mathrm{Ext}^{q}(E_{-p},A) \otimes F_{-p}=
H^{q+k_{-p}}(\mathcal E_{-p}\otimes A) \otimes \mathcal F_{-p}
\]
which is functorial in $A$ and converges to
\begin{equation}
E_{\infty}^{p,q}= \left\{
\begin{array}{cc}
A & \textrm{\quad if $p+q=0$} \\
0 & \textrm{\quad otherwise.}
\end{array}
\right.
\end{equation}
\end{theorem}

\section{Cohomology of Ulrich bundles on $\PP^2\times\PP^2$}\label{sec3}

 Let $X=\PP^2\times\PP^2$. We will often use the following exact sequences.
\begin{equation}\label{a1}0\to\Oo_X(0,-3)\to\Oo_X(0,-2)^{\oplus 3}\to \Oo_{\PP^2}\boxtimes\Omega\to 0,\end{equation}

\begin{equation}\label{a2}0\to\Oo_{\PP^2}\boxtimes\Omega\to\Oo_X(0,-1)^{\oplus 3}\to\Oo_X\to 0,\end{equation}

\begin{equation}\label{a3}0\to\Oo_X(-3,0)\to\Oo_X(-2,0)^{\oplus 3}\to \Omega\boxtimes\Oo_{\PP^2}\to 0,\end{equation} and

\begin{equation}\label{a4}0\to\Omega\boxtimes\Oo_{\PP^2}\to\Oo_X(-1,0)^{\oplus 3}\to\Oo_X\to 0,\end{equation}

On $X$ we have two different notions of Castelnuovo-Mumford regularity
(see \cite{bm2} and \cite{hw}):

\begin{definition}\label{reg1}
  A  coherent sheaf $\sF$ on $X$ is
  $(BM)$-regular if:
  $$H^1(\sF(-1,0))=H^1(\sF(0,-1))=H^2(\sF(-1,-1))=H^2(\sF(0,-2))=H^2(\sF(-2,0))=$$ $$=H^3(\sF(-1,-2))=H^3(\sF(-2,-1))=H^4(\sF(-2,-2))=0.$$
\end{definition}

\begin{definition}\label{reg2}
  A  coherent sheaf $\sF$ on $X$ is
  $(HW)$-regular if:
  $$H^1(\sF(-1,0))=H^1(\sF(-1,-1))=H^2(\sF(-2,-1))=H^2(\sF(-1,-2))=$$ $$=H^3(\sF(-3,-1))=H^3(\sF(-1,-3))=H^3(\sF(-2,-2))=$$ $$=H^4(\sF(-1,-4))
  =H^4(\sF(-4,-1))=H^4(\sF(-3,-2))=H^4(\sF(-2,-3))=0.$$
\end{definition}

\begin{remark}
If $\sF$ is a $(BM)$-regular coherent on $X$ then it is globally generated and $\sF(p,p')$ is $(BM)$-regular for $p,p'\geq 0$ by \cite{bm2} Proposition 2.2.\\

If $\sF$ is a $(HW)$-regular coherent on $X$ then it is globally generated and $\sF(p,p')$ is $(HW)$-regular for $p,p'\geq 0$ by \cite{hw} Proposition 2.7 and Proposition 2.8.

\end{remark}

\begin{lemma}
\label{riv}
Let $\Vv$ be an Ulrich bundle on $X$.
\begin{enumerate}
\item[(i)] $H^0(\Vv(j_1,j_2))=0$ for $j_1\leq -1$, $\j_2\leq -1$ and $H^4(\Vv(j_1,j_2))=0$ for $j_1\geq -4$, $\j_2\geq -4$.
\item[(ii)] $\Vv$ is $(BM)$-regular and $(HW)$-regular.
\item[(iii)] For $i=1,2,3$, $H^i(\Vv(j_1,j_2))=0$ if $j_1\geq -i$ and $\j_2\geq -i$ or if $j_1\leq -i-1$ and $\j_2\leq -i-1$.
\item[(iV)] For $i=1,2,3$, $H^i(\Vv\otimes\Omega(j_1+2)\boxtimes\Oo_{\PP^2}(j_2))=H^i(\Vv\otimes\Oo_{\PP^2}(j_1)\boxtimes\Omega(j_2+2))=0$ if $j_1\geq -i$ and $\j_2\geq -i$ or if $j_1\leq -i-1$ and $\j_2\leq -i-1$.
\end{enumerate}
\end{lemma}
\begin{proof}
$H^i(\Vv(j,j))=0$ for any $i$ and $j=-1,-2,-3,-4$. Since $H^0(\Vv(-1,-1))=0$,  from (\ref{a1}), (\ref{a2}), (\ref{a3}) and (\ref{a4}), we get $H^0(\Vv(j_1,j_2))=0$ for $j_1\leq -1$, $\j_2\leq -1$, $H^0(\Vv\otimes\Omega(j_1+1)\boxtimes\Oo_{\PP^2}(j_2))=H^0(\Vv\otimes\Oo_{\PP^2}(j_1)\boxtimes\Omega(j_2+1))=0$ for $j_1\leq -1$, $\j_2\leq -1$.\\

Since $H^4(\Vv(-4,-4))=0$,  from (\ref{a1}), (\ref{a2}), (\ref{a3}) and (\ref{a4}) we get $H^4(\Vv(j_1,j_2))=0$ for $j_1\geq -4$, $\j_2\geq -4$, $H^4(\Vv\otimes\Omega(j_1+2)\boxtimes\Oo_{\PP^2}(j_2))=H^4(\Vv\otimes\Oo_{\PP^2}(j_1)\boxtimes\Omega(j_2+2))=0$ for $j_1\geq -4$, $\j_2\geq -4$.
In particular we get $(i)$.\\

Since $H^3(\Vv(-3,-3))=0$, from (\ref{a2}) tensored by $\Vv(-3,-2)$ and (\ref{a4}) tensored by $\Vv(-2,-3)$, $H^3(\Vv(-3,-2))=H^3(\Vv(-2,-3))=0$.
From (\ref{a2}) tensored by $\Vv(-3,-1)$ and (\ref{a4}) tensored by $\Vv(-1,-3)$, $H^3(\Vv(-3,-1))=H^3(\Vv(-1,-3))=0$.
Since $H^3(\Vv(-2,-2))=0$, from (\ref{a2}) tensored by $\Vv(-2,-1)$ and (\ref{a4}) tensored by $\Vv(-1,-2)$, $H^3(\Vv(-2,-1))=H^3(\Vv(-1,-2))=0$.

 %Moreover we may obtain $H^3(\Vv(j_1,j_2))=0$ for $j_1\geq -3$, $\j_2\geq -3$.

  From  (\ref{a1}) tensored by $\Vv(-3,-1)$ and (\ref{a3}) tensored by $\Vv(-1,-3)$ we get $H^3(\Vv\otimes\Omega(-1)\boxtimes\Oo_{\PP^2}(-3))=H^3(\Vv\otimes\Oo_{\PP^2}(-3)\boxtimes\Omega(-1))=0$.

 Since $H^3(\Vv(-3,-2))=H^3(\Vv(-2,-3))=0$, from  (\ref{a1}) tensored by $\Vv(-2,-1)$ and (\ref{a3}) tensored by $\Vv(-1,-2)$ we get $H^3(\Vv\otimes\Omega(-1)\boxtimes\Oo_{\PP^2}(-2))=H^3(\Vv\otimes\Oo_{\PP^2}(-2)\boxtimes\Omega(-1))=0$.

 Since $H^3(\Vv(-2,-2))=0$, from  (\ref{a1}) tensored by $\Vv(-2,0)$ and (\ref{a3}) tensored by $\Vv(0,-2)$ we get $H^3(\Vv\otimes\Omega\boxtimes\Oo_{\PP^2}(-2))=H^3(\Vv\otimes\Oo_{\PP^2}(-2)\boxtimes\Omega)=0$.

%for $j_1\geq -3$, $\j_2\leq -3$.\\

Since $H^2(\Vv(-2,-2))=0$ and  $H^3(\Vv\otimes\Omega(-1)\boxtimes\Oo_{\PP^2}(-2))=H^3(\Vv\otimes\Oo_{\PP^2}(-2)\boxtimes\Omega(-1))=0$, from (\ref{a2}) tensored by $\Vv(-2,-1)$ and (\ref{a4}) tensored by $\Vv(-1,-2)$, $H^2(\Vv(-2,-1))=H^2(\Vv(-1,-2))=0$.

Since $H^2(\Vv(-2,-1))=H^2(\Vv(-2,-1))=0$ and  $H^3(\Vv\otimes\Omega\boxtimes\Oo_{\PP^2}(-2))=H^3(\Vv\otimes\Oo_{\PP^2}(-2)\boxtimes\Omega)=0$, from (\ref{a2}) tensored by $\Vv(-2,0)$ and (\ref{a4}) tensored by $\Vv(0,-2)$ $H^2(\Vv(-2,0))=H^2(\Vv(0,-2))=0$.

 %Moreover we may obtain $H^2(\Vv(j_1,j_2))=0$ for $j_1\geq -2$, $\j_2\geq -2$. From  (\ref{a1}) tensored by $\Vv(-2,0)$ and (\ref{a3}) tensored by $\Vv(0,-2)$ we get %$H^2(\Vv\otimes\Omega(j_1+2)\boxtimes\Oo_{\PP^2}(j_2))=H^2(\Vv\otimes\Oo_{\PP^2}(j_1)\boxtimes\Omega(j_2+2))=0$ for $j_1\geq -2$, $\j_2\leq -2$.\\

Since $H^3(\Vv(-3,-1))=H^3(\Vv(-1,-3))=H^2(\Vv(-2,-1))=H^2(\Vv(-2,-1))=0$, from  (\ref{a1}) tensored by $\Vv(-1,0)$ and (\ref{a3}) tensored by $\Vv(0,-1)$ we get $H^2(\Vv\otimes\Omega\boxtimes\Oo_{\PP^2}(-1))=H^2(\Vv\otimes\Oo_{\PP^2}(-1)\boxtimes\Omega)=0$.

Since $H^1(\Vv(-1,-1))=0$ and $H^2(\Vv\otimes\Omega\boxtimes\Oo_{\PP^2}(-1))=H^2(\Vv\otimes\Oo_{\PP^2}(-1)\boxtimes\Omega)=0$, from (\ref{a2}) tensored by $\Vv(-1,0)$ and (\ref{a4}) tensored by $\Vv(0,-1)$ $H^1(\Vv(-1,0))=H^1(\Vv(0,-1))=0$.\\

 %Moreover we may obtain $H^1(\Vv(j_1,j_2))=0$ for $j_1\geq -1$, $\j_2\geq -1$. From  (\ref{a1}) tensored by $\Vv(-1,1)$ and (\ref{a3}) tensored by $\Vv(1,-1)$ we get %$H^1(\Vv\otimes\Omega(j_1+2)\boxtimes\Oo_{\PP^2}(j_2))=H^1(\Vv\otimes\Oo_{\PP^2}(j_1)\boxtimes\Omega(j_2+2))=0$ for $j_1\geq -1$, $\j_2\leq -1$.\\

We have hence proved the $(BM)$-regularity and $(HW)$-regularity so $(ii)$.\\

Since $\Vv(p,p')$ is $(BM)$-regular and $(HW)$-regular for $p,p'\geq 0$, we obtain $H^i(\Vv(j_1,j_2))=0$ if $j_1\geq -i$ and $\j_2\geq -i$ for  $i>0$.

Now since $\Vv^\vee(2,2)$ is Ulrich we have $H^j(\Vv^\vee(k_1+2,k_2+2))=0$ if $k_1\geq -j$ and $k_2\geq -j$, so by Serre duality
$H^{4-j}(\Vv(-k_1-5,-k_2-5))=0$ if $k_1\geq -j$ and $k_2\geq -j$. Let $i=4-j$, $j_1=-k_1-5$, $j_2=-k_2-5$, we get $H^i(\Vv(j_1,j_2))=0$ if $j_1\leq -i-1$ and $j_2\leq -i-1$ for $i=1,2,3$. So also $(iii)$ is proved.\\

Now from $$0\to\Vv(j_1,j_2-1)\to\Vv(j_1,j_2)^{\oplus 3}\to \Vv\otimes\Oo_{\PP^2}(j_1)\boxtimes\Omega(j_2+2)\to 0$$ we get $H^i(\Vv\otimes\Oo_{\PP^2}(j_1)\boxtimes\Omega(j_2+2))=0$ if $j_1\geq -i$ and $\j_2\geq -i$.

From $$0\to\Vv\otimes\Oo_{\PP^2}(j_1)\boxtimes\Omega(j_2+1)\to\Vv(j_1,j_2)^{\oplus 3}\to\Vv(j_1,j_2+1)\to 0$$ we get $H^i(\Vv\otimes\Oo_{\PP^2}(j_1)\boxtimes\Omega(j_2+2))=0$  if $j_1\leq -i-1$ and $\j_2\leq -i-1$ for $i=1,2,3$,. In the same way we obtain $H^i(\Vv\otimes\Omega(j_1+2)\boxtimes\Oo_{\PP^2}(j_2))=0$ if $j_1\geq -i$ and $\j_2\geq -i$ or if $j_1\leq -i-1$ and $\j_2\leq -i-1$ for $i=1,2,3$, and proof of $(iv)$ is completed.
\end{proof}

\begin{remark}

\begin{multicols}{3}
[Let call for $i=0,\dots 4$
$$\textrm{$a_i=h^i(\Vv(-i-1,-i))$ and $b_i=h^i(\Vv(-i,-i-1)).$}$$
An helpful tool for summarizing the part $(iii)$ of the above Lemma are the following pictures dealing with the subsets of the plane $j_1,j_2$ whose points correspond to some intermediate cohomology group of the sheaf $\Vv(j_1,j_2)$.\\
 In the following pictures we show the vanishing regions and the positions of $a_i, b_i$ for $h^1(\Vv(j_1,j_2)), h^2(\Vv(j_1,j_2))$ and $h^3(\Vv(j_1,j_2))$:
]
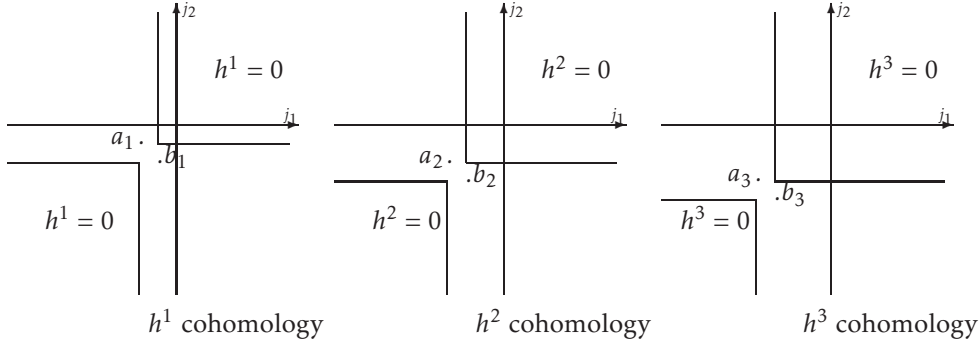
\begin{center}

\centerline{
\setlength{\unitlength}{0.25cm}
\begin{picture}(20,20)(-10,-12)
\put(6,0){\vector(1,0){0.5}}
\put(5.7,0.3){\tiny$j_1$}
\put(0,6){\vector(0,1){0.5}}
\put(0.3,6){\tiny$j_2$}
\put(0,-9){\line(0,6){15}}
\put(-9,0){\line(6,0){15}}
\put(-2,-1){$.$}
\put(-3.5,-1){$a_1$}
\put(-1,-2){$.b_1$}
\put(2,2.5){$h^1=0$}
\put(-2,-2){\line(0,-1){7}}
\put(-2,-2){\line(-1,0){7}}
\put(-1,-1){\line(0,1){7}}
\put(-1,-1){\line(1,0){7}}
\put(-7,-5.5){$h^1=0$}
\put(-1.5,-11){\rm{$h^1$ cohomology}}
\end{picture}}
\columnbreak

\centerline{
\setlength{\unitlength}{0.25cm}
\begin{picture}(20,20)(-10,-12)
\put(6,0){\vector(1,0){0.5}}
\put(5.7,0.3){\tiny$j_1$}
\put(0,6){\vector(0,1){0.5}}
\put(0.3,6){\tiny$j_2$}
\put(0,-9){\line(0,6){15}}
\put(-9,0){\line(6,0){15}}
\put(-3,-2){$.$}
\put(-4.5,-2){$a_2$}
\put(-2,-3){$.b_2$}
\put(2,2.5){$h^2=0$}
\put(-3,-3){\line(0,-1){6}}
\put(-3,-3){\line(-1,0){6}}
\put(-2,-2){\line(0,1){8}}
\put(-2,-2){\line(1,0){8}}
\put(-7,-5.5){$h^2=0$}
\put(-1.5,-11){\rm{$h^2$ cohomology}}
\end{picture}}

\centerline{
\setlength{\unitlength}{0.25cm}
\begin{picture}(20,20)(-10,-12)
\put(6,0){\vector(1,0){0.5}}
\put(5.7,0.3){\tiny$j_1$}
\put(0,6){\vector(0,1){0.5}}
\put(0.3,6){\tiny$j_2$}
\put(0,-9){\line(0,6){15}}
\put(-9,0){\line(6,0){15}}
\put(-4,-3){$.$}
\put(-5.5,-3){$a_3$}
\put(-3,-4){$.b_3$}
\put(2,2.5){$h^3=0$}
\put(-4,-4){\line(0,-1){5}}
\put(-4,-4){\line(-1,0){5}}
\put(-3,-3){\line(0,1){9}}
\put(-3,-3){\line(1,0){9}}
\put(-8,-5.5){$h^3=0$}
\put(-1.5,-11){\rm{$h^3$ cohomology}}
\end{picture}}
\end{center}
\end{multicols}
\end{remark}

 \begin{lemma}
\label{fam2}
Let $\Vv$ be an Ulrich bundle on $X$, then \begin{enumerate}
%\item $h^i(\Vv(-2,-1))=h^i(\Vv(-1,-2))=0$ for any $i$ except for $i=1$.
\item[(a)] $h^1(\Vv\otimes\Omega\boxtimes\Oo_{\PP^2}(-1))=b_0$, $h^1(\Vv\otimes\Oo_{\PP^2}(-1)\boxtimes\Omega))=a_0$ and $h^i(\Vv\otimes\Omega\boxtimes\Oo_{\PP^2}(-1))=h^i(\Vv\otimes\Oo_{\PP^2}(-1)\boxtimes\Omega))=0$ for any $i$ except for $i=1$.
\item[(b)] $h^1(\Vv\otimes\Omega\boxtimes\Oo_{\PP^2}(-2))=b_2$, $h^1(\Vv\otimes\Oo_{\PP^2}(-2)\boxtimes\Omega)=a_2$ and $h^i(\Vv\otimes\Omega\boxtimes\Oo_{\PP^2}(-2))=h^i(\Vv\otimes\Oo_{\PP^2}(-2)\boxtimes\Omega))=0$ for any $i$ except for $i=1$.
%\item $h^2(\Vv(-2,-3))=y$, $h^2(\Vv(-3,-2))=x$ and $h^i(\Vv(-3,-2))=h^i(\Vv(-2,-3))=0$ for any $i$ except for $i=2$.
\item[(c)] $h^1(\Vv\otimes\Omega\boxtimes\Omega)=3a_2+3b_2$ and $h^i(\Vv\otimes\Omega\boxtimes\Omega)=0$ for any $i$ except for $i=1$.
\item[(d)] $h^2(\Vv\otimes\Omega(-1)\boxtimes\Omega(-1))=3a_3+3b_3=3a_1+3b_1$ and $h^i(\Vv\otimes\Omega(-1)\boxtimes\Omega(-1))=0$ for any $i$ except for $i=2$.
\end{enumerate}
\end{lemma}
\begin{proof}

 From  (\ref{a2}) tensored by $\Vv(-1,0)$ and (\ref{a4}) tensored by $\Vv(0,-1)$ we get $h^1(\Vv\otimes\Oo_{\PP^2}(-1)\boxtimes\Omega)=a_0$, $h^1(\Vv\otimes\Omega\boxtimes\Oo_{\PP^2}(-3))=b_0$ and $h^i(\Vv\otimes\Oo_{\PP^2}(-1)\boxtimes\Omega))=h^i(\Vv\otimes\Omega\boxtimes\Oo_{\PP^2}(-1))=0$ for any $i$ except for $i=1$. We have proved $(a)$.\\

 From  (\ref{a1}) tensored by $\Vv(-2,0)$ and (\ref{a3}) tensored by $\Vv(0,-2)$ we get $h^1(\Vv\otimes\Omega\boxtimes\Oo_{\PP^2}(-2))=b_2$, $h^1(\Vv\otimes\Oo_{\PP^2}(-2)\boxtimes\Omega)=a_2$ and $h^i(\Vv\otimes\Omega\boxtimes\Oo_{\PP^2}(-2))=h^i(\Vv\otimes\Oo_{\PP^2}(-2)\boxtimes\Omega)=0$ for any $i$ except for $i=1$. We have proved $(b)$.\\

 From  (\ref{a1}) tensored by $\Vv(-3,0)$ and (\ref{a3}) tensored by $\Vv(0,-3)$ we get $h^2(\Vv\otimes\Oo_{\PP^2}(-3)\boxtimes\Omega)=3b_2$, $h^2(\Vv\otimes\Omega\boxtimes\Oo_{\PP^2}(-3))=3a_2$ and $h^i(\Vv\otimes\Oo_{\PP^2}(-3)\boxtimes\Omega))=h^i(\Vv\otimes\Omega\boxtimes\Oo_{\PP^2}(-3))=0$ for any $i$ except for $i=2$.\\

 Now let us consider the sequence

% \begin{equation}\label{a5}0\to \Omega\boxtimes\Omega\to\Oo_{\PP^2}(-1)^{\oplus 3}\boxtimes\Omega\to \Oo_{\PP^2}\boxtimes\Omega\to 0,\end{equation}
 \begin{equation}\label{a5}0\to \Oo_{\PP^2}(-3)\boxtimes\Omega\to \Oo_{\PP^2}(-2)^{\oplus 3}\boxtimes\Omega\to \Omega\boxtimes\Omega\to 0,\end{equation}
tensored by $\Vv$. Since we have computed the cohomology of $\Vv\otimes\Oo_{\PP^2}(-3)\boxtimes\Omega$ and $\Vv\otimes\Oo_{\PP^2}(-2)\boxtimes\Omega$, we may deduce that $h^1(\Vv\otimes\Omega\boxtimes\Omega)=3a_2+3b_2$ and $h^i(\Vv\otimes\Omega\boxtimes\Omega)=0$ for any $i$ except for $i=1$. We have proved $(c)$.\\

From  (\ref{a1}) tensored by $\Vv(-4,-1)$ and (\ref{a3}) tensored by $\Vv(-1,-4)$ we get $h^3(\Vv\otimes\Oo_{\PP^2}(-4)\boxtimes\Omega(-1))=3b_3$, $h^3(\Vv\otimes\Omega(-1)\boxtimes\Oo_{\PP^2}(-4))=3a_3$ and $h^i(\Vv\otimes\Oo_{\PP^2}(-4)\boxtimes\Omega(-1))=h^i(\Vv\otimes\Omega(-1)\boxtimes\Oo_{\PP^2}(-4))=0$ for any $i$ except for $i=3$.\\

 From  (\ref{a1}) tensored by $\Vv(-3,-1)$ and (\ref{a3}) tensored by $\Vv(-1,-3)$ we get $h^2(\Vv\otimes\Omega(-1)\boxtimes\Oo_{\PP^2}(-3))=b_3$, $h^2(\Vv\otimes\Oo_{\PP^2}(-3)\boxtimes\Omega(-1))=a_3$ and $h^i(\Vv\otimes\Omega(-1)\boxtimes\Oo_{\PP^2}(-3))=h^i(\Vv\otimes\Oo_{\PP^2}(-3)\boxtimes\Omega(-1))=0$ for any $i$ except for $i=2$.

 Now let us consider the sequence

% \begin{equation}\label{a5}0\to \Omega\boxtimes\Omega\to\Oo_{\PP^2}(-1)^{\oplus 3}\boxtimes\Omega\to \Oo_{\PP^2}\boxtimes\Omega\to 0,\end{equation}
 \begin{equation}\label{a6}0\to \Oo_{\PP^2}(-4)\boxtimes\Omega(-1)\to \Oo_{\PP^2}(-3)^{\oplus 3}\boxtimes\Omega(-1)\to \Omega(-1)\boxtimes\Omega(-1)\to 0,\end{equation}
tensored by $\Vv$. Since we have computed the cohomology of $\Vv\otimes\Oo_{\PP^2}(-4)\boxtimes\Omega(-1)$ and $\Vv\otimes\Oo_{\PP^2}(-3)\boxtimes\Omega(-1)$, we may deduce that $h^2(\Vv\otimes\Omega\boxtimes\Omega)=3a_3+3b_3$ and $h^i(\Vv\otimes\Omega(-1)\boxtimes\Omega(-1))=0$ for any $i$ except for $i=2$.

From  (\ref{a2}) and (\ref{a4}) tensored by $\Vv(-1,-1)$   we get $h^1(\Vv\otimes\Oo_{\PP^2}(-1)\boxtimes\Omega(-1))=3b_1$, $h^1(\Vv\otimes\Omega(-1)\boxtimes\Oo_{\PP^2}(-1))=3a_1$ and $h^i(\Vv\otimes\Oo_{\PP^2}(-1)\boxtimes\Omega(-1))=h^i(\Vv\otimes\Omega(-1)\boxtimes\Oo_{\PP^2}(-1))=0$ for any $i$ except for $i=1$.\\

 From  (\ref{a2}) tensored by $\Vv(-2,-1)$ and (\ref{a4}) tensored by $\Vv(-1,-2)$ we get $h^2(\Vv\otimes\Omega(-1)\boxtimes\Oo_{\PP^2}(-2))=a_1$, $h^2(\Vv\otimes\Oo_{\PP^2}(-2)\boxtimes\Omega(-1))=b_1$ and $h^i(\Vv\otimes\Omega(-1)\boxtimes\Oo_{\PP^2}(-2))=h^i(\Vv\otimes\Oo_{\PP^2}(-2)\boxtimes\Omega(-1))=0$ for any $i$ except for $i=2$.

 Now let us consider the sequence

\begin{equation}\label{a7}0\to \Omega(-1)\boxtimes\Omega(-1)\to\Oo_{\PP^2}(-2)^{\oplus 3}\boxtimes\Omega(-1)\to \Oo_{\PP^2}(-1)\boxtimes\Omega(-1)\to 0,\end{equation}
 %\begin{equation}\label{a6}0\to \Oo_{\PP^2}(-4)\boxtimes\Omega(-1)\to \Oo_{\PP^2}(-3)^{\oplus 3}\boxtimes\Omega(-1)\to \Omega(-1)\boxtimes\Omega(-1)\to 0,\end{equation}
tensored by $\Vv$. Since we have computed the cohomology of $\Vv\otimes\Oo_{\PP^2}(-2)\boxtimes\Omega(-1)$ and $\Vv\otimes\Oo_{\PP^2}(-1)\boxtimes\Omega(-1)$, we may deduce that $h^2(\Vv\otimes\Omega\boxtimes\Omega)=3a_1+3b_1$ and $h^i(\Vv\otimes\Omega(-1)\boxtimes\Omega(-1))=0$ for any $i$ except for $i=2$. So also $(d)$ is proved.

\end{proof}

\section{Families of Ulrich bundles on $\PP^2\times\PP^2$}\label{sec4}
 We start this section with examples of Ulrich bundles:

 \begin{remark}
The only rank one Ulrich bundles on $X$ are $\Oo_X(2,0)$ and $\Oo_X(0,2)$.
\end{remark}

\begin{example}
Let $a\geq 1$ and $b\geq a+2$ then the set of elements $$\Phi:\Oo_{\PP^2}^b\to\Oo_{\PP^2}(1)^a$$ such that $H^0(\Phi(1))$ are surjective forms a non-empty dense open subset (see \cite{EH} Proposition 4.1). Using the vector bundles obtained as the kernel of these maps when $b=2a$ and $a>1$ in \cite{CMP} Theorem 3.6. has been constructed families of Ulrich of even rank. The construction works also for odd rank. So for any $r>1$ we have families of rank $r$ Ulrich bundles arising from the following exact sequences.
\begin{equation}\label{e1}0\to\Oo_X(-1,1)^{\oplus r}\to\Oo_X(0,1)^{\oplus 2r}\to \Vv_1\to 0,\end{equation} or

\begin{equation}\label{e2}0\to\Oo_X(1,-1)^{\oplus r}\to\Oo_X(1,0)^{\oplus 2r}\to \Vv_2\to 0.\end{equation}
 The same families may be obtained from the exact sequences

 \begin{equation}\label{e3}0\to\Vv_1\to\Oo_X(1,0)^{\oplus 2r}\to\Oo_X(2,0)^{\oplus r}\to 0,\end{equation} or

\begin{equation}\label{e4}0\to\Vv_2\to\Oo_X(0,1)^{\oplus 2r}\to\Oo_X(0,2)^{\oplus r}\to 0.\end{equation}

The same families may be given by the exact sequences

\begin{equation}\label{e5}0\to\Oo_X(1,0)^{\oplus r}\to\Oo_{\PP^2}(1)\boxtimes\Omega(2)^{\oplus r}\to \Vv_1\to 0,\end{equation} or

\begin{equation}\label{e6}0\to\Oo_X(0,1)^{\oplus r}\to\Omega(2)\boxtimes\Oo_{\PP^2}(1)^{\oplus r}\to \Vv_2\to 0.\end{equation}
\end{example}

\begin{lemma}
\label{fam}
Let $a,b,c,d$ be integers. Let $\Vv_1$ and $\Vv_1$ be  indecomposable coherent sheaves on $X$ arising from  exact sequences

\begin{equation}\label{u1}
0\to\Oo_X(-1,1)^{\oplus d}\to\Oo_X(0,1)^{\oplus b}\to \Vv_1\to 0,\end{equation} or

\begin{equation}\label{u2}0\to\Oo_X(1,-1)^{\oplus c}\to\Oo_X(1,0)^{\oplus a}\to \Vv_2\to 0.\end{equation}
Then we have:\begin{enumerate}
\item $Ext^1(\Vv_1,\Vv_2)=Ext^1(\Vv_2,\Vv_1)=0$.
\item If $\Vv_1$ $\Vv_2$ are Ulrich bundles we must have $a=2c$ and $b=2d$.
\end{enumerate}

\end{lemma}
\begin{proof}$(1)$ If we apply the functor $Hom(-,\Vv_2)$ to (\ref{u1}) we obtain $Ext^1(\Vv_1,\Vv_2)=0$ because $Hom(\Oo_X(-1,1),\Vv_2)=H^0(\Vv_2(1,-1))=0$ and $Ext^1(\Oo_X(0,1),\Vv_2)=H^1(\Vv_2(0,-1))=0$. Similarly we prove that $Ext^1(\Vv_2,\Vv_1)=0$.

$(2)$ The rank of $\Vv_1$ is $b-d$ so if $\Vv_1$ is Ulrich we must have $h^0(\Vv_1)=6b-6d$. From (\ref{u1}) we get $h^0(\Vv_1)=3b$, hence $3b=6b-6d$ if and only if $b=2d$.

\end{proof}

Now we construct the full exceptional collections that we will use in the next theorems:
 Let us consider on both copies of $\PP^2$ the full exceptional collection $\{\Oo_{\PP^2}(-2), \Oo_{\PP^2}(-1), \Oo_{\PP^2}\}$. We may obtain (see \cite{Orlov}):
\begin{equation}\label{col}\begin{aligned}
&\Ee_8[k_8]=\Oo_X(-2,-2)[-4]~,~ \Ee_7[k_7]=\Oo_X(-1,-2)[-4]~,~ \Ee_6[k_6]=\Oo_X(-2,-1)[-3], \\
&\Ee_5[k_5]=\Oo_X(0, -2)[-3]~,~ \Ee_4[k_4]=\Oo_X(-2,0)[-2]~,~\Ee_3[k_3]=\Oo_X(-1,-1)[-1], \\
&\Ee_2[k_2]=\Oo_X(0,-1)[-1]~,~ \Ee_1[k_1]=\Oo_X(-1,0)~,~ \Ee_0[k_0]=\Oo_X.
\end{aligned}
\end{equation}
The associated full exceptional collection $\langle F_8=\mathcal F_n, \ldots, F_0=\mathcal F_0\rangle$ of Theorem \ref{use} is

\begin{equation}\label{cold}\begin{aligned}
&\Ff_8=\Oo_X(-1,-1)~,~ \Ff_7=\Omega(1)\boxtimes\Oo_{\PP^2}(-1)~,~ \Ff_6=\Oo_{\PP^2}(-1)\boxtimes\Omega(1), \\
&\Ff_5=\Oo_X(0,-1)~,~ \Ff_4=\Oo_X(-1,0)~,~\Ff_3=\Omega(1)\boxtimes\Omega(1), \\
&\Ff_2=\Oo_{\PP^2}\boxtimes\Omega(1)~,~ \Ff_1=\Omega(1)\boxtimes\Oo_{\PP^2}~,~ \Ff_0=\Oo_X.
\end{aligned}\end{equation}

From (\ref{col}) with a few left mutations we obtain:

\begin{equation}\label{col2}
\begin{aligned}
&\Ee_8[k_8]=\Oo_X(-2,-2)[-4]~,~ \Ee_7[k_7]=\Oo_{\PP^2}(-2)\boxtimes\Omega[-4]~,~ \Ee_6[k_6]=\Omega\boxtimes\Oo_{\PP^2}(-2)[-3], \\
&\Ee_5[k_5]=\Oo_X(-2,-1)[-3]~,~ \Ee_4[k_4]=\Oo_X(-1,-1)[-2]~,~\Ee_3[k_3]=\Oo_X(-1,-1)[-1], \\
&\Ee_2[k_2]=\Oo_X(0,-1)[-1]~,~ \Ee_1[k_1]=\Oo_X(-1,0)~,~ \Ee_0[k_0]=\Oo_X.
\end{aligned}
\end{equation}
The associated full exceptional collection $\langle F_8=\mathcal F_n, \ldots, F_0=\mathcal F_0\rangle$ of Theorem \ref{use} is

\begin{equation}\label{cold2}\begin{aligned}
&\Ff_8=\Oo_X(-1,-1)~,~ \Ff_7=\Omega(1)\boxtimes\Oo_{\PP^2}(-1)~,~ \Ff_6=\Oo_{\PP^2}(-1)\boxtimes\Omega(1), \\
&\Ff_5=\Oo_X(0,-1)~,~ \Ff_4=\Oo_X(-1,0)~,~\Ff_3=\Omega(1)\boxtimes\Omega(1), \\
&\Ff_2=\Oo_{\PP^2}\boxtimes\Omega(1)~,~ \Ff_1=\Omega(1)\boxtimes\Oo_{\PP^2}~,~ \Ff_0=\Oo_X.
\end{aligned}\end{equation}

Let call $\Gg_1=\Oo_{\PP^2}\boxtimes\Omega(1)$ and $\Gg_2=\Omega(1)\boxtimes\Oo_{\PP^2}$.

\begin{theorem}\label{volon}
Let $\Vv$ be an Ulrich bundle on $X$. Then $\Vv$ arises from an exact sequence of the form:

\begin{equation}\label{res}
0\to\Oo_X(-1,0)^{\oplus a_2}\oplus\Oo_X(0,-1)^{\oplus b_2}\to\Oo_X(-1,1)^{\oplus a_1}\oplus\Oo_X^{\oplus 3a_2+3b_2}\oplus\Oo_X(1,-1)^{\oplus b_1}\to\end{equation}
$$\to\Oo_X(1,0)^{\oplus a_0}\oplus\Oo_X(0,1)^{\oplus b_0}\to \Vv\to 0.$$
or \begin{equation}\label{resd}
0\to\Vv\to\Oo_X(1,2)^{\oplus b_4}\oplus\Oo_X(2,1)^{\oplus a_4}\to\Oo_X(1,3)^{\oplus a_3}\oplus\Oo_X(2,2)^{\oplus 3a_2+3b_2}\oplus\Oo_X(3,1)^{\oplus b_3}\to\end{equation}
$$\to\Oo_X(2,3)^{\oplus a_2}\oplus\Oo_X(3,2)^{\oplus b_2}\to 0.$$
\end{theorem}
\begin{proof}
We consider the Beilinson type spectral sequence associated to $\Aa:=\Vv(-1,-1)$ and identify the members of the graded sheaf associated to the induced filtration as the sheaves mentioned in the statement. We assume due to \cite[Proposition 2.1]{ES} that
$$H^i(\Aa(-j,-j))=0 \text{ for all $i$ and $0\le j \le 3$}$$
and consider the full exceptional collection $\Ee_{\bullet}$ given in (\ref{col2}) and  collection $\Ff_{\bullet}$ given in (\ref{cold2}).

 We construct a Beilinson complex, quasi-isomorphic to $\Aa$, by calculating $H^{i+k_j}(\Aa\otimes \Ff_j)\otimes \Ee_j$ with  $i,j \in \{0, \ldots, 8\}$ to get the following table. Here we use several vanishing in the intermediate cohomology of $\Aa, \Aa(-1,-1),\Aa(-2,-2),  \Aa(-3,-3)$ together with vanishing of Lemma \ref{fam2}:

\begin{center}\begin{tabular}{|c|c|c|c|c|c|c|c|c|}
\hline
$\Oo_X(-2,-2)$ & $\Oo_X(-1,-2) $& $\Oo_X(-2,-1)$ & $\Oo_X(0,-2)$ & $\Oo_X(-2,0)$ & $\Oo_X(-1,-1)$ & $\Oo_X(0,-1)$ & $\Oo_X(-1,0)$ & $\Oo_X$ \\
\hline
\hline
$0$        &$0$        &$*$         &	 $*$	 	& $*$		 	& $*$		 	& $*$			& $*$	&$*$\\
$0$        &$0$    &$0$		    &	 $0$	& $*$	& $*$		 	& $*$			& $*$	&$*$\\
$0$        &$0$    &$0$	        &	 $0$	 	& $0$	    & $*$		    & $*$			& $*$	&$*$\\
$0$        &$b_2$  &$0$	        &	 $0$	& $0$	 	& $0$		 	& $0$			& $*$	&$*$\\

$0$        &$0$        &$a_2$		    &	 $b_1$	 	& $0$	 	    & $0$		 	& $0$		& $0$	&$0$\\
$*$        &$*$        &$0$		    &	 $0$	 	& $a_1$	 	    & $0$		 	& $0$			& $0$	&$0$\\
$*$        &$*$        &$*$		    &	 $*$	 	& $0$	 	    & $3a_2+3b_2$		 	& $a_0$		& $0$	&$0$\\

$*$        &$*$        &$*$		    &	 $*$	 	& $*$		 	& $0$		 	& $0$			& $b_0$	&$0$\\
$*$        &$*$        &$*$		    &	 $*$	 	& $*$	 	& $*$		 	& $*$			& $0$	&$0$\\

\hline
\hline
$\Oo_X(-1,-1)$& $\Gg_2(0,-1)$ & $\Gg_1(-1,0)$ & $\Oo_X(0,-1)$ & $\Oo_X(-1,0)$ & $\Gg_1\otimes\Gg_2$
&$\Gg_1$ &$\Gg_2$&$\Oo_X$ \\

\hline
\end{tabular}
\end{center}
From this table, since $Ext^k(F_i,F_j)=0$ for $k>0$ and any $i,j$,  we have that the full exceptional collection (\ref{col22}) is strong. So we get the claimed resolution.
\end{proof}
\begin{remark}
From (\ref{res}) we deduce that we must have $a_0\not=0$ or $b_0\not=0$. From (\ref{resd}) we deduce that we must have $a_4\not=0$ or $b_4\not=0$.
\end{remark}
\begin{corollary}\label{volon2}
Let $\Vv$ be an Ulrich bundle on $X$ with $H^1(\Vv\otimes\Omega\boxtimes\Omega)=0$. Then $\Vv$ arises from an exact sequence of the form:

\begin{equation}\label{ww} 0\to\Oo_X(-1,1)^{\oplus a_1}\to\Oo_X(0,1)^{\oplus 2a_1}\to \Vv\to 0,\end{equation} or

\begin{equation}\label{ww2}0\to\Oo_X(1,-1)^{\oplus b_1}\to\Oo_X(1,0)^{\oplus 2b_1}\to \Vv\to 0,\end{equation}
with $a_1,b_1>1$.
\end{corollary}
\begin{proof}
We consider the Beilinson type spectral sequence associated to $\Aa:=\Vv(-1,-1)$ with the full exceptional collection $\Ee_{\bullet}$ and left dual collection $\Ff_{\bullet}$ as in the above Theorem.

Since $H^1(\Vv\otimes\Omega\boxtimes\Omega)=0$ and $H^1(\Vv\otimes\Omega\boxtimes\Omega)=3a_2+3b_2$ we get $$a_2=b_2=0,$$
so we obtain the following table:

\begin{center}\begin{tabular}{|c|c|c|c|c|c|c|c|c|}
\hline
$\Oo_X(-2,-2)$ & $\Oo_X(-1,-2) $& $\Oo_X(-2,-1)$ & $\Oo_X(0,-2)$ & $\Oo_X(-2,0)$ & $\Oo_X(-1,-1)$ & $\Oo_X(0,-1)$ & $\Oo_X(-1,0)$ & $\Oo_X$ \\
\hline
\hline
$0$        &$0$        &$*$         &	 $*$	 	& $*$		 	& $*$		 	& $*$			& $*$	&$*$\\
$0$        &$0$    &$0$		    &	 $0$	& $*$	& $*$		 	& $*$			& $*$	&$*$\\
$0$        &$0$    &$0$	        &	 $0$	 	& $0$	    & $*$		    & $*$			& $*$	&$*$\\
$0$        &$0$  &$0$	        &	 $0$	& $0$	 	& $0$		 	& $0$			& $*$	&$*$\\

$0$        &$0$        &$0$		    &	 $b_1$	 	& $0$	 	    & $0$		 	& $0$		& $0$	&$0$\\
$*$        &$*$        &$0$		    &	 $0$	 	& $a_1$	 	    & $0$		 	& $0$			& $0$	&$0$\\
$*$        &$*$        &$*$		    &	 $*$	 	& $0$	 	    & $0$		 	& $a_0$		& $0$	&$0$\\

$*$        &$*$        &$*$		    &	 $*$	 	& $*$		 	& $0$		 	& $0$			& $b_0$	&$0$\\
$*$        &$*$        &$*$		    &	 $*$	 	& $*$	 	& $*$		 	& $*$			& $0$	&$0$\\

\hline
\hline
$\Oo_X(-1,-1)$& $\Gg_2(0,-1)$ & $\Gg_1(-1,0)$ & $\Oo_X(0,-1)$ & $\Oo_X(-1,0)$ & $\Gg_1\otimes\Gg_2$
&$\Gg_1$ &$\Gg_2$&$\Oo_X$ \\

\hline
\end{tabular}
\end{center}

We call $\alpha$ and $\beta$ the maps arising from the table which are defined from $\Oo_X(-2,0)^{a_1}$ to $\Oo_X(-1,0)^{a_0}$ and from $\Oo_X(0,-2)^{b_1}$ to $\Oo_X(0,-1)^{b_0}$. So we get the following exact sequence

$$0\to\ker\alpha\to\Oo_X(-2,0)^{a_1}\to \Oo_X(-1,0)^{a_0}\to coker\alpha\to 0$$
and
$$0\to\ker\beta\to\Oo_X(0,-2)^{b_1}\to \Oo_X(0,-1)^{b_0}\to coker\beta\to 0,$$
where $\ker\alpha\cong \Bb\boxtimes\Oo_{\PP^2}$ and $coker\beta\cong \Oo_{\PP^2}\boxtimes\Cc$ with $\Bb$ a vector bundle on $\PP^2$ and $\Cc$ a coherent sheaf on $\PP^2$. Moreover $\ker\beta=0$ since the spectral sequence converges to an object in degree $0$, so we obtain that $h^0(\Cc)=0$, hence $Hom(\ker\alpha,coker\beta)\cong H^0(\Bb^\vee\boxtimes\Cc)=0$. This implies that also $\ker\alpha=0$ and $\Aa$ is given by an extension of $coker\alpha$ with $coker\beta$ But, by Lemma\ref{fam}, $Ext^1(coker\beta,coker\alpha)=0$ and $a_0=2a_1$ and $b_0=2b_1$. Then we have the claimed result.
\end{proof}

\begin{theorem}\label{volon3}
Let $\Vv$ be an Ulrich bundle on $X$ with $H^2(\Vv\otimes\Omega(-1)\boxtimes\Omega(-1))=0$. Then $\Vv\cong\Oo_X(2,0)$  or $\Vv\cong\Oo_X(0,2)$.
\end{theorem}
\begin{proof}
We consider the Beilinson type spectral sequence associated to $\Aa:=\Vv(-1,-1)$ and identify the members of the graded sheaf associated to the induced filtration as the sheaves mentioned in the statement. We assume due to \cite[Proposition 2.1]{ES} that
$$H^i(\Aa(-j,-j))=0 \text{ for all $i$ and $0\le j \le 3$}$$
and consider the full exceptional collection $\Ee_{\bullet}$ given in (\ref{col2}) and  collection $\Ff_{\bullet}$ given in (\ref{cold2}).

We construct a Beilinson complex, quasi-isomorphic to $\Aa$, by calculating $H^{i+k_j}(\Aa\otimes \Ee_j)\otimes \Ff_j$ with  $i,j \in \{0, \ldots, 8\}$ to get the following table. Here we use several vanishing in the intermediate cohomology of $\Aa, \Aa(-1,-1),\Aa(-2,-2),  \Aa(-3,-3)$ together with vanishing of Lemma \ref{fam2}:

\begin{center}\begin{tabular}{|c|c|c|c|c|c|c|c|c|}
\hline
$\Oo_X(-1,-1) $ & $\Oo_X(-1,0)$ & $\Oo_X(0,-1)$
& $\Oo_X(-1,1)$ & $\Oo_X(1,-1)$ & $\Gg_1\boxtimes\Gg_2$ & $\Gg_1$ & $\Gg_2$ & $\Oo_X$ \\
\hline
\hline
$0$        &$0$        &$*$         &	 $*$	 	& $*$		 	& $*$		 	& $*$			& $*$	&$*$\\
$0$        &$0$    &$0$		    &	 $0$	& $*$	& $*$		 	& $*$			& $*$	&$*$\\
$0$        &$a_3$    &$0$	        &	 $0$	 	& $0$	    & $*$		    & $*$			& $*$	&$*$\\
$0$        &$0$  &$b_3$	        &	 $a_2$	& $0$	 	& $0$		 	& $0$			& $*$	&$*$\\

$0$        &$0$        &$0$		    &	 $0$	 	& $b_2$	 	    & $0$		 	& $0$		& $0$	&$0$\\
$*$        &$*$        &$0$		    &	 $0$	 	& $0$	 	    & $0$		 	& $0$			& $0$	&$0$\\
$*$        &$*$        &$*$		    &	 $*$	 	& $0$	 	    & $0$		 	& $b_1$		& $0$	&$0$\\

$*$        &$*$        &$*$		    &	 $*$	 	& $*$		 	& $0$		 	& $0$			& $a_1$	&$0$\\
$*$        &$*$        &$*$		    &	 $*$	 	& $*$	 	& $*$		 	& $*$			& $0$	&$0$\\

\hline
\hline
$\Oo_X(-2,-2)$ & $\Gg_1(-2,-1)$ & $\Gg_2(-1,-2)$ & $\Oo_X(-2,-1)$ & $\Oo_X(-1,-1)$  & $\Oo_X(-1,-1)$ & $\Oo_X(0,-1) $ & $\Oo_X(-1,0)$ & $\Oo_X$ \\

\hline
\end{tabular}
\end{center}

Since $h^2(\Vv\otimes\Omega(-1)\boxtimes\Omega(-1))=3a_3+3b_3=3a_1+3b_1=0$ we get $$a_3=b_3=a_1=b_1=0,$$
so we obtain the following table:

\begin{center}\begin{tabular}{|c|c|c|c|c|c|c|c|c|}
\hline
$\Oo_X(-1,-1) $ & $\Oo_X(-1,0)$ & $\Oo_X(0,-1)$
& $\Oo_X(-1,1)$ & $\Oo_X(1,-1)$ & $\Gg_1\boxtimes\Gg_2$ & $\Gg_1$ & $\Gg_2$ & $\Oo_X$ \\
\hline
\hline
$0$        &$0$        &$*$         &	 $*$	 	& $*$		 	& $*$		 	& $*$			& $*$	&$*$\\
$0$        &$0$    &$0$		    &	 $0$	& $*$	& $*$		 	& $*$			& $*$	&$*$\\
$0$        &$0$    &$0$	        &	 $0$	 	& $0$	    & $*$		    & $*$			& $*$	&$*$\\
$0$        &$0$  &$0$	        &	 $a_2$	& $0$	 	& $0$		 	& $0$			& $*$	&$*$\\

$0$        &$0$        &$0$		    &	 $0$	 	& $b_2$	 	    & $0$		 	& $0$		& $0$	&$0$\\
$*$        &$*$        &$0$		    &	 $0$	 	& $0$	 	    & $0$		 	& $0$			& $0$	&$0$\\
$*$        &$*$        &$*$		    &	 $*$	 	& $0$	 	    & $0$		 	& $0$		& $0$	&$0$\\

$*$        &$*$        &$*$		    &	 $*$	 	& $*$		 	& $0$		 	& $0$			& $0$	&$0$\\
$*$        &$*$        &$*$		    &	 $*$	 	& $*$	 	& $*$		 	& $*$			& $0$	&$0$\\

\hline
\hline
$\Oo_X(-2,-2)$ & $\Gg_1(-2,-1)$ & $\Gg_2(-1,-2)$ & $\Oo_X(-2,-1)$ & $\Oo_X(-1,-1)$  & $\Oo_X(-1,-1)$ & $\Oo_X(0,-1) $ & $\Oo_X(-1,0)$ & $\Oo_X$ \\

\hline
\end{tabular}
\end{center}
Finally, since $Ext^1(\Oo_X(-1,1),\Oo_X(1,-1))=Ext^1(\Oo_X(1,-1), \Oo_X(-1,1))=0$ we get the claimed result.
\end{proof}

\begin{theorem}

Let $\Vv$ be an Ulrich bundle on $X$, then \begin{enumerate}

\item if $a_0=a_4=0$, then $\Vv\cong\Omega(2)\boxtimes\Omega(3)$ or $\Vv\cong\Oo_X(0,2)$;
\item  if $b_0=b_4=0$, then $\Vv\cong\Omega(3)\boxtimes\Omega(2)$ or $\Vv\cong\Oo_X(2,0)$;
\item  if $\Vv$ is a twist of a pullback from $\PP^2$ then $\Vv\cong\Oo_X(2,0)$, or $\Vv\cong\Oo_X(0,2)$ or $\Vv$ arises from sequences (\ref{ww}), (\ref{ww2}).

\end{enumerate}
\end{theorem}
\begin{proof} We prove only $(2)$ and $(3)$.

Since an Ulrich bundles  $\Vv$ arising from $$0\to\Oo_X(1,-1)^{\oplus b_1}\to\Oo_X(1,0)^{\oplus 2b_1}\to \Vv\to 0,$$ or $$0\to\Oo_X(-1,1)^{\oplus a_1}\to\Oo_X(0,1)^{\oplus 2a_1}\to \Vv\to 0,$$do not satisfy the condition $b_4=0$ or $b_0=0$ we may assume $H^1(\Vv\otimes\Omega\boxtimes\Omega)=3a_2+3b_2\not=0$ so $a_2=h^2(\Vv^\vee(0,-1))=h^2(\Vv(-3,-2))\not=0$ or $a_2=h^2(\Vv^\vee(0,-1))=h^2(\Vv(-3,-2))\not=0$.

If $a_2\not=0$, from the exact sequence $$0\to\Oo(-3,-2)\to\Oo(-2,-2)^{\oplus 3}\to\Omega\boxtimes\Oo(-1)^{\otimes 3}\to\Omega\boxtimes\Omega(1)\to 0$$ tensored by $\Vv$, since $h^2(\Vv(-2,-2))=0$ and $h^1(\Vv\otimes\Omega\boxtimes\Oo(-1))=b_0=0$ we get a surjection from $H^0(\Vv\otimes \Omega\boxtimes\Omega(1))=H^0(\Vv\otimes(\Omega(2)\boxtimes\Omega(3))^\vee)$ to $H^2(\Vv(-3,-2))\not=0$. Moreover from the exact sequence
 $$0\to\Oo(0,-1)\to\Oo^{\oplus 3}\to\Oo(1)\boxtimes\Omega(2)^{\otimes 3}\to\Omega(3)\boxtimes\Omega(2)\to 0$$ tensored by $\Vv^\vee$, since $h^2(\Vv^\vee)=0$ and $h^1(\Vv^\vee\otimes\Oo(1)\boxtimes\Omega(2))=h^3(\Vv\otimes\Oo(1)\boxtimes\Omega(-2))=b_4=0$ we get a surjection from $H^0(\Vv^\vee\otimes \Omega(3)\boxtimes\Omega(2))$ to $H^2(\Vv^\vee(0,-1))\not=0$.

 So we may conclude that  $\Vv\cong \Omega(3)\boxtimes\Omega(2)$ by arguing as in \cite{bm2}.

 If $a_2=0$ then (\ref{res}) becomes $$
0\to\Oo_X(0,-1)^{\oplus b_2}\to\Oo_X(-1,1)^{\oplus a_1}\oplus\Oo_X^{\oplus 3b_2}\oplus\Oo_X(1,-1)^{\oplus b_1}\to\Oo_X(1,0)^{\oplus a_0}\to \Vv\to 0,$$ and we deduce that $a_1=0$.

Similarly (\ref{resd}) becomes
$$0\to\Vv\to\Oo_X(2,1)^{\oplus a_4}\to\Oo_X(1,3)^{\oplus a_3}\oplus\Oo_X(2,2)^{\oplus 3b_2}\oplus\Oo_X(3,1)^{\oplus b_3}\to\Oo_X(3,2)^{\oplus b_2}\to 0$$ and we deduce that $a_3=0$.

Finally let assume  $b_2\not=0$ and $a_1=a_3=0$, from the exact sequence $$0\to\Oo(-2,-3)\to\Oo(-2,-2)^{\oplus 3}\to\Oo(-2,-1)^{\oplus 3}\to\Oo(-2,0)\to 0$$ tensored by $\Vv$, since $h^2(\Vv(-2,-2))=0$ and $h^1(\Vv(-2,-1))=a_1=0$ we get a surjection from $H^0(\Vv(-2,0))=H^0(\Vv\otimes(\Oo(2,0))^\vee)$ to $H^2(\Vv(-2,-3))\not=0$. Moreover from the exact sequence
 $$0\to\Oo(-1,0)\to\Oo^{\oplus 3}\to\Oo(1,0)^{\oplus 3}\to\Oo(2,0)\to 0$$ tensored by $\Vv^\vee$, since $h^2(\Vv^\vee)=0$ and $h^1(\Vv^\vee(1,0))=h^3(\Vv(-4,-3))=a_3=0$ we get a surjection from $H^0(\Vv^\vee(2,0))$ to $H^2(\Vv^\vee(-1,0))\not=0$.

 So we may conclude that  $\Vv\cong \Oo(2,0)$ (see \cite{bm2}) and $(2)$ is proven.

In order to prove $(3)$ let assume that $\Vv$ is a twist of a pullback from the second copy of $\PP^2$.

First let us consider the case $\Vv=\Oo_{\PP^2}\boxtimes\Bb$ where $\Bb$ is a vector bundle on $\PP^2$. Since $h^i(\Vv(-3))=h^i(\Vv(-4))=0$ for any $i$ but $h^2(\Oo_{\PP^2}(-3))\not=0$ and $h^2(\Oo_{\PP^2}(-4))\not=0$ we must have $h^i(\Bb(-3))=h^i(\Bb(-4))=0$ for any $i$. So we may deduce that $\Bb(-2)$ is Ulrich on $\PP^2$ hence $\Vv=\Oo_X(0,2)$.

Now let us consider the case $\Vv=\Oo_{\PP^2}(1)\boxtimes\Bb$ where $\Bb$ is a vector bundle on $\PP^2$. Since $h^i(\Vv(-1))=h^i(\Vv(-4))=0$ for any $i$ but $h^0(\Oo_{\PP^2})\not=0$ and $h^2(\Oo_{\PP^2}(-3))\not=0$  we must have $h^i(\Bb(-1))=0$ and $h^i(\Bb(-4))=0$ for any $i$. In particular $h^2(\Bb(t))=0$ for any $t\geq -4$. We consider the Beilinson type spectral sequence associated to $\Bb(-1)$
with  $\Ee_{\bullet}=\{\Oo_{\PP^2}(-1), \Omega(1), \Oo_{\PP^2}\}$  and  $\Ff_{\bullet}=\{\Oo_{\PP^2}(-2), \Oo_{\PP^2}(-1), \Oo_{\PP^2}\}$ given in (\ref{cold2}).

We get the following table.

\begin{center}\begin{tabular}{|c|c|c|}
\hline
$\Oo_{\PP^2}(-2)$ & $\Oo_{\PP^2}(-1)$ &$\Oo_{\PP^2}$ \\
\hline
\hline

 $0$		& $0$	&$0$\\

 $b$			& $a$	&$0$\\
 $0$			& $0$	&$0$\\

\hline
\hline
$\Oo_{\PP^2}(-1)$ & $\Omega(1)$ &$\Oo_{\PP^2}$ \\

\hline
\end{tabular}
\end{center}
So we obtain $$0\to \Oo_{\PP^2}(-2)^{\oplus b}\to\Oo_{\PP^2}(-1)^{\oplus a}\to\Bb(-1)\to 0,$$
hence $$0\to\Oo_X(1,-1)^{\oplus b}\to\Oo_X(1,0)^{\oplus a}\to \Vv_2\to 0.$$ Then by Lemma \ref{u1} we get  $a=2b$ and we obtain (\ref{ww2}).

Finally let us consider the case $\Vv=\Oo_{\PP^2}(t)\boxtimes\Bb$ where $\Bb$ is a vector bundle on $\PP^2$ and $t\geq 2$. Since $h^0(\Vv(-1))=0$ we must have $h^i(\Bb(-1))=0$ and we may deduce that $b_0=0$. Since $h^2(\Oo_{\PP^2}(t-4))=0$  we deduce that $b_4=0$. So by $(2)$ we obtain $\Vv=\Oo_X(2,0)$.

\end{proof}

\begin{remark} We have just proved that the Ulrich bundles obtained as a pullback from $\PP^2$ are rigid or they are in the high dimensional families  defined in \cite{CMP}. For instance the closure of the family of rank two bundles $\Mm_2$ in the associated moduli space is a generically smooth component of dimension $5$ (see \cite{CMP} Theorem 3.9.).

$\Omega(3)\boxtimes\Omega(2)$ and $\Omega(2)\boxtimes\Omega(3)$ are the only Ulrich bundles characterized so far which are not pullbacks.\\
Another cohomological characterization of $\Omega(3)\boxtimes\Omega(2)$ and $\Omega(2)\boxtimes\Omega(3)$ can be found in \cite{MM}.
\end{remark}

So far we have seen and characterized families of Ulrich bundles with $a_0=0$ or $b_0=0$. Now we construct three interesting families with both $a_0\not=0$ and $b_0\not=0$:

\begin{example}\label{quas0}
Since  $Ext^1(\Oo(2,0),\Omega(2)\boxtimes\Omega(3))\cong H^1(\Omega\boxtimes\Omega(3))\cong \mathbb C^8 $ we get a $7$-dimensional familiy of rank $5$ Ulrich bundles arising from the following extension
\begin{equation}\label{aes}
0\to\Omega(2)\boxtimes\Omega(3)\to\Vv\to\Oo(2,0)\to 0.
\end{equation}
We get $a_0=h^0(\Vv(-1,0))=3$ and $b_0=h^0(\Vv(0,-1))=9$.\\
Let us consider the dual of (\ref{aes}) tensored by $\Vv$
$$0\to\Vv(-2,0)\to\Vv^\vee\otimes\Vv\to\Omega(1)\boxtimes\Omega\otimes\Vv\to 0.$$
 From sequence (\ref{aes}) tensored by $\Oo(-2,0)$ we get $H^0(\Vv(-2,0))=0$ if the extension (\ref{aes}) is not trivial.\\
 From sequence (\ref{aes}) tensored by $\Omega(1)\boxtimes\Omega$ we get $$h^0(\Omega(1)\boxtimes\Omega\otimes\Vv)=h^0(\Omega(2)\boxtimes\Omega(3)\otimes\Omega(1)\boxtimes\Omega)=1.$$
 So we obtain $h^0(\Vv\otimes\Vv^\vee)=1$, hence $\Vv$ is simple.

 From these rank $5$ simple Ulrich bundles we may obtain higher rank Ulrich bundles by other extensions with $\Oo(2,0)$ or with $\Omega(2)\boxtimes\Omega(3)$.
\end{example}

\begin{example}\label{quas00}
Since  $Ext^1(\Omega(3)\boxtimes\Omega(2),\Omega(2)\boxtimes\Omega(3))\cong H^1(\Omega\otimes\Omega(2)\boxtimes\Omega(3)\otimes\Omega(1))\cong \mathbb C^{27} $ we get a $26$-dimensional family of rank $8$ Ulrich bundles arising from the following extension
\begin{equation}\label{aes2}
0\to\Omega(2)\boxtimes\Omega(3)\to\Vv\to\Omega(3)\boxtimes\Omega(2)\to 0.
\end{equation}
We get $a_0=h^0(\Vv(-1,0))=9$ and $b_0=h^0(\Vv(0,-1))=9$\\
Let us consider the dual of (\ref{aes2}) tensored by $\Vv$
$$0\to\Omega\boxtimes\Omega(1)\otimes\Vv\to\Vv^\vee\otimes\Vv\to\Omega(1)\boxtimes\Omega\otimes\Vv\to 0.$$
 From sequence (\ref{aes2}) tensored by $\Omega\boxtimes\Omega(1)$ we get $H^0(\Omega\boxtimes\Omega(1)\otimes\Vv)=0$ if the extension (\ref{aes2}) is not trivial.\\
 From sequence (\ref{aes2}) tensored by $\Omega(1)\boxtimes\Omega$ we get $$h^0(\Omega(1)\boxtimes\Omega\otimes\Vv)=h^0(\Omega(2)\boxtimes\Omega(3)\otimes\Omega(1)\boxtimes\Omega)=1.$$
 So we obtain $h^0(\Vv\otimes\Vv^\vee)=1$, hence $\Vv$ is simple.

 From these rank $8$ simple Ulrich bundles we may obtain higher rank Ulrich bundles by other extension with $\Omega(3)\boxtimes\Omega(2)$ or with $\Omega(2)\boxtimes\Omega(3)$.
\end{example}
Now we have a more complicated example arising from the following proposition:

\begin{proposition}\label{quas}
The generic element $\psi\in Ext^1(\Oo\boxtimes\Omega(2),\Oo(1)\boxtimes\Omega)$ gives an extension

\begin{equation}\label{ab}
0\to\Oo(1)\boxtimes\Omega\to\Ee_{\psi}\to\Oo\boxtimes\Omega(2)\to 0
\end{equation}
where $\mathcal{E}_{\psi}(1)$ is a simple Ulrich bundle of rank $4$ for which $a_0\not=0$ and $b_0\not=0$.
\end{proposition}
\begin{proof}

Let  $W=Ext^1(\Omega(2),\Omega)\cong \mathbb C^3 $. Let us notice that any nonzero $\eta\in W$ gives an exact sequence
$$0\to\Omega\to\Oo_{\mathbb P^2}^{\oplus 2}\oplus\Oo_{\mathbb P^2}(1)^{\oplus 2}\to\Omega(2)\to 0$$ and the map in cohomology
$$\delta_{\eta}: H^0(\Omega(2))\to H^1(\Omega)$$ is nonzero.

On $\mathbb P(W^\vee)\times \mathbb P^2$ we have $Ext^1(\Oo\boxtimes\Omega(2),\Oo(1)\boxtimes\Omega)\cong H^1(\Oo(1)\boxtimes(\Omega(2))^\vee)\cong H^0(\mathbb P^2, \Oo(1))\otimes H^1(\mathbb P^2,\Omega\otimes(\Omega(2))^\vee)\cong W^\vee\times W$. Now let us consider the identity element $I\in W^\vee\times W$ ($I$ restricts to $\eta$ on $\mathbb P(W^\vee)$) and we obtain

$$0\to\Oo(1)\boxtimes\Omega\to\Ee_{I}\to\Oo\boxtimes\Omega(2)\to 0.$$ In cohomology we get the map
$$H^0(\Oo\boxtimes\Omega(2))\xrightarrow{\delta_{I}} H^1(\Oo(1)\boxtimes\Omega)$$ where $$H^1(\Oo(1)\boxtimes\Omega)\cong H^0(\mathbb P^2, \Oo(1))\times H^1(\mathbb P^2,\Omega)=W^\vee\times H^1(\mathbb P^2,\Omega)$$
and  $$H^0(\Oo\boxtimes\Omega(2))\cong H^0(\mathbb P^2,\Omega(2)).$$ We may conclude that $\delta_I$ is equivalent to
$$\phi: W\to H^0(\mathbb P^2,\Omega(2))^\vee \times H^1(\mathbb P^2,\Omega).$$ Since $\phi(\eta)=\delta_\eta\not= 0$,  $\phi$ must be an inclusion. But $\dim(H^0(\mathbb P^2,\Omega(2))^\vee \times H^1(\mathbb P^2,\Omega))=\dim(W)=3$ so $\phi$ is an isomorphism.

Then we have showed that the generic element $\psi\in Ext^1(\Oo\boxtimes\Omega(2),\Oo(1)\boxtimes\Omega)$ gives an extension (\ref{ab}) with $$H^0(\Oo\boxtimes\Omega(2))\cong H^1(\Oo(1)\boxtimes\Omega).$$ So we obtain $H^0(\Ee_\psi)=H^1(\Ee_\psi)=0$. Moreover we compute $$H^i(\Ee_\psi)=H^i(\Ee_\psi(-1))=H^i(\Ee_\psi(-2))=H^i(\Ee_\psi(-2))=0$$ for any $i$. Then $\Ee_\psi(1)$ is Ulrich with both $a_0\not=0$ and $b_0\not=0$.

From the dual of (\ref{ab}) it is easy to check that $h^1(\Ee_\psi\otimes\Ee_\psi^\vee)=1$, hence $\Ee_\psi$ is simple.
\end{proof}

\section{Ulrich bundles on the flag variety $F(0,1,2)$}\label{sec5}

Let $F\subseteq\mathbb P^7$ be the del Pezzo threefold  of degree $6$ and Picard number two. Let us consider $F$ as an hyperplane section of $\mathbb P^2\times\mathbb P^2$ with the two natural projections $p_i:F\subset\mathbb P^2\times\mathbb P^2\to\mathbb P^2$ and the following rank two vector bundles:

$$
\Gg_1=p_1^*\Omega_{\mathbb P^2}^1(h_1)\qquad \Gg_2=p_2^*\Omega_{\mathbb P^2}^1(h_2),
$$

We may consider the full exceptional collection
\begin{equation}\label{col11} \{E_5=\Oo_F(-1,-1)[-2], E_4 = \Gg_2(-1,-1)[-2], E_3 = \Gg_1(-1,-1)[-1],\end{equation}
$$E_2 = \Oo_F(-1,0)[-1] , E_1 = \Oo_F(0,-1), E_0 = \Oo_F\}$$

and
\begin{equation}\label{col22} \{F_0 = \Oo_F, F_1 = \Gg_2(0,-1), F_2 = \Gg_1(-1,0), F_3 =  \Oo_F(0,-1), F_4 = \Oo_F(-1,0), F_5 = \Oo_F(-1,-1)\}
\end{equation}

\begin{theorem}\label{volon44}
Let $\Vv$ be an Ulrich bundle on $F$.  Then $\Vv$ arises from an exact sequence of the form:
\begin{equation}\label{res22}
0\to\Oo_F(0,1)^{\oplus c}\oplus\Oo_F(1,0)^{\oplus d}\to\Gg_1(1,1)^{\oplus b}\oplus\Gg_2(1,1)^{\oplus a}\to\Vv\to 0.\end{equation}

\end{theorem}
\begin{proof}
We consider the Beilinson type spectral sequence associated to $\Aa:=\Vv(-1,-1)$ and identify the members of the graded sheaf associated to the induced filtration as the sheaves mentioned in the statement. We  consider the full exceptional collection $\Ee_{\bullet}$ and left dual collection $\Ff_{\bullet}$ in (\ref{col11}) and (\ref{col22}).

 We construct a Beilinson complex, quasi-isomorphic to $\Aa$, by calculating $H^{i+k_j}(\Aa\otimes \Ff_j)\otimes \Ee_j$ with  $i,j \in \{0, \ldots, 6\}$ to get the following table:

 \begin{center}\begin{tabular}{|c|c|c|c|c|c|c|c|c|c|c|}
\hline
 $\Oo_F(-1,-1)$ & $\Oo_F(-1,0)$ & $\Oo_F(0,-1)$ & $\Gg_1$ & $\Gg_2$ & $\Oo_F$ \\
 \hline
 \hline
$H^3$	&	$H^3$	&	$*$		&	$*$		&	$*$		&	$*$	\\
$H^2$	&	$H^2$	&	$H^3$	&	$H^3$	&	$*$		&	$*$	\\
$H^1$	& 	$H^1$	&	$H^2$	&	$H^2$	&	$H^3$	&	$H^3$	\\
$H^0$	& 	$H^0$	&	$H^1$	&	$H^1$	&	$H^2$	&	$H^2$	\\
$*$		&	$*$ 	 	&	$H^0$	&	$H^0$	& 	$H^1$	& 	$H^1$	\\
$*$		&	$*$		&	$*$		&	$*$		&	$H^0$	& 	$H^0$ \\
\hline
$\Aa(-1,-1)$		& $\Aa\otimes\Gg_2(-1,0)$	 & $\Aa \otimes \Gg_1(0,-1)$	 & $\Aa(-1,0)$		& $\Aa(0,-1)$		&$\Aa$\\
\hline
\end{tabular}
\end{center}

We assume due to \cite[Proposition 2.1]{ES} that
$$H^i(\Aa(-j,-j))=0 \text{ for all $i$ and $0\le j \le 2$}.$$
From the exact sequence
$$0\to\Aa\otimes\Gg_1(-2,-1)\to\Aa(-1,-1)^{\oplus 3}\to\Aa(0,-1)\to 0,$$
since $H^3(\Aa\otimes\Gg_1(-2,-1))=0$ we get $H^2(\Aa(0,-1))=0$. In a similar way we get $H^2(\Aa(-1,0))=0$. So the table become
\begin{center}\begin{tabular}{|c|c|c|c|c|c|c|c|c|c|c|}
\hline
 $\Oo_F(-1,-1)$ & $\Oo_F(-1,0)$ & $\Oo_F(0,-1)$ & $\Gg_1$ & $\Gg_2$ & $\Oo_F$ \\
 \hline
 \hline
$0$	&	$0$	&	$*$		&	$*$		&	$*$		&	$*$	\\
$0$	&	$f$	&	$0$	&	$0$	&	$*$		&	$*$		\\
$0$	& 	$c$	&	$e$	&	$0$	&	$0$	&	$0$	\\
$0$	& 	$0$	&	$d$	&	$b$	&	$0$	&	$0$	\\
$*$		&	$*$		 	 	&	$0$	&	$0$	& 	$a$	& 	$0$	\\
$*$		&	$*$		&	$*$		&	$*$			&	$0$	& 	$0$ \\
\hline
$\Aa(-1,-1)$		& $\Aa\otimes\Gg_2(-1,0)$	 & $\Aa \otimes \Gg_1(0,-1)$	 & $\Aa(-1,0)$		& $\Aa(0,-1)$		&$\Aa$\\
\hline
\end{tabular}
\end{center}
Since  the spectral sequence converges to an object in degree $0$ we get $e=f=0$, so
\begin{center}\begin{tabular}{|c|c|c|c|c|c|c|c|c|c|c|}
\hline
 $\Oo_F(-1,-1)$ & $\Oo_F(-1,0)$ & $\Oo_F(0,-1)$ & $\Gg_1$ & $\Gg_2$ & $\Oo_F$ \\
 \hline
 \hline
$0$	&	$0$	&	$*$		&	$*$		&	$*$		&	$*$	\\
$0$	&	$0$	&	$0$	&	$0$	&	$*$		&	$*$		\\
$0$	& 	$c$	&	$0$	&	$0$	&	$0$	&	$0$	\\
$0$	& 	$0$	&	$d$	&	$b$	&	$0$	&	$0$	\\
$*$		&	$*$		 	&	$0$	&	$0$	& 	$a$	& 	$0$	\\
$*$		&	$*$		&	$*$		&	$*$			&	$0$	& 	$0$ \\
\hline
$\Aa(-1,-1)$		& $\Aa\otimes\Gg_2(-1,0)$	 & $\Aa \otimes \Gg_1(0,-1)$	 & $\Aa(-1,0)$		& $\Aa(0,-1)$		&$\Aa$\\
\hline
\end{tabular}
\end{center}
Finally since $Ext^i(\Oo_F(-1,0),\Oo_F(0,-1))=0$, $Ext^i(\Oo_F(-1,0),\Gg_1)=0$, $Ext^i(\Oo_F(-1,0),\Gg_2)=0$, $Ext^i(\Oo_F(0,-1),\Gg_1)=0$, $Ext^i(\Oo_F(0,-1),\Gg_2)=0$ and $Ext^i(\Gg_1,\Gg_2)=0$ for any $i>0$ we have that the full exceptional collection (\ref{col22}) is strong. So we get the claimed resolution.
\end{proof}

\begin{remark}\begin{enumerate}Let $\Vv$ an indecomposable Ulrich bundle on $F$.
 \item If $c=h^1(\Vv\otimes\Gg_2(0,1))=0$ and $b=h^1(\Vv(0,1))=0$ then $\Vv$ is the restriction of a bundle arising from sequence (\ref{e5}).\\
If $d=h^1(\Vv\otimes\Gg_1(1,0))=0$ and $a=h^1(\Vv(1,0))=0$ then $\Vv$ is the restriction of a bundle arising from sequence (\ref{e6}).
\item If $d=b=0$ we obtain the exact sequence $$0\to\Oo_F(0,1)\to\Gg_2(1,1)\to\Oo_F(2,0)\to 0$$ so $\Vv\cong\Oo_F(2,0)$.\\
If $c=a=0$ we obtain $\Vv\cong\Oo_F(0,2)$.
\end{enumerate}
\end{remark}

%%%%%%%%%%%%%%%%%%%%%%%%%%%%%%%%%%%%%

\providecommand{\bysame}{\leavevmode\hbox to3em{\hrulefill}\thinspace}
\providecommand{\MR}{\relax\ifhmode\unskip\space\fi MR }
% \MRhref is called by the amsart/book/proc definition of \MR.
\providecommand{\MRhref}[2]{%
  \href{http://www.ams.org/mathscinet-getitem?mr=#1}{#2}
}
\providecommand{\href}[2]{#2}

\end{document}